\documentclass{amsart}

\usepackage[all]{xy}

\usepackage{latexsym}

\usepackage{amssymb}
\usepackage{amsfonts}
\usepackage{amscd}
\usepackage{amsmath,amsthm}

\newtheorem{lemma}{Lemma}[section]
\newtheorem{proposition}[lemma]{Proposition}
\newtheorem{theorem}[lemma]{Theorem}
\newtheorem{corollary}[lemma]{Corollary}

\newtheorem{prop}[lemma]{Proposition}

\newtheorem{cor}[lemma]{Corollary}

{

}

\theoremstyle{definition}

\newtheorem{definition}[lemma]{\sl Definition}

\theoremstyle{remark}

\newtheorem{remark}[lemma]{Remark}

\newcommand{\Hom}{\operatorname{Hom}}
\newcommand{\dlim}{\underset{n \rightarrow \infty}{\lim}}

\newcommand{\ra}{\operatorname{\rightarrow}}

\newcommand{\uTor}{\underline{\mathcal{T}{\it{or}}}}

\newcommand{\usHomt}{{\underline{{\mathcal{H}}{\it{om}}}}_{\widetilde{A^{op}}}}
\newcommand{\sHomt}{{\mathcal{H}{\it{om}}}_{\widetilde{A^{op}}}}
\newcommand{\sHom}{{\mathcal{H}}{\it{om}}_{A}}
\newcommand{\sHomb}{{\mathcal{H}}{\it{om}}_{B}}
\newcommand{\sHomc}{{\mathcal{H}}{\it{om}}_{C}}
\newcommand{\usHom}{{\underline{{\mathcal{H}}{\it{om}}}}_{A}}
\newcommand{\usHomao}{{\underline{{\mathcal{H}}{\it{om}}}}_{\widetilde{A^{op}}}}
\newcommand{\usHomb}{{\underline{{\mathcal{H}}{\it{om}}}}_{B}}

\newcommand{\usHomc}{{\underline{{\mathcal{H}}{\it{om}}}}_{C}}
\newcommand{\usHomk}{{\underline{{\mathcal{H}}{\it{om}}}}_{K}}
\newcommand{\sExt}{{\mathcal{E}xt}_{A}}
\newcommand{\sExtb}{{\mathcal{E}xt}_{B}}
\newcommand{\usExt}{{\underline{\mathcal{E}xt}}_{A}}
\newcommand{\usExtb}{{\underline{\mathcal{E}xt}}_{B}}
\newcommand{\intotimes}{{\underline{\otimes}}_{A}}
\newcommand{\intotimesb}{{\underline{\otimes}}_{B}}

\numberwithin{equation}{section}

\begin{document}

\title{Local Duality for Connected $\mathbb{Z}$-algebras}
\author{I. Mori}
\address{Shizuoka University}
\email{mori.izuru@shizuoka.ac.jp}
\author{A. Nyman}
\address{Western Washington University}
\email{adam.nyman@wwu.edu}

\thanks{The first author was supported by JSPS Grant-in-Aid for Scientific Research (C) 16K05097 and JSPS Grant-in-Aid for Scientific Research (B) 16H03923.}

\begin{abstract}
We develop basic homological machinery for $\mathbb{Z}$-algebras in order to prove a version of local duality for Ext-finite connected $\mathbb{Z}$-algebras.  As an application, we compare two notions of regularity for such algebras.
\end{abstract}

\maketitle

\pagenumbering{arabic}

Throughout this paper, $k$ will denote a field over which all objects are defined.

\section{Introduction}
In algebraic geometry, projective varieties over $k$ are studied via their homogenous coordinate rings, which are commutative, connected, finitely generated $\mathbb{N}$-graded $k$-algebras.  This simple observation inspired some of the early development in noncommutative algebraic geometry, in which geometric objects are defined in terms of (possibly noncommutative) $\mathbb{Z}$-graded rings, specializing to the usual commutative theory \cite{az1}, \cite{vervekind}.  The subject has since blossomed, finding points of contact with physics, differential geometry and number theory.

In another approach to the subject, initially explored in \cite{bp} and \cite{polish} and further developed in \cite{quadrics} and \cite{sierra}, geometric objects are constructed from $\mathbb{Z}$-algebras, which are defined in terms of pre-additive categories whose objects are indexed by $\mathbb{Z}$.  From this perspective, the $\mathbb{Z}$-algebra associated to a pre-additive category is its ring of morphisms.  For example, if $X$ is a projective variety embedding in a projective space and $\mathcal{O}(i)$ denotes the $i$-th tensor power of the tautological line bundle, one defines $D_{ij} = \Hom_{\mathcal{O}_X}(\mathcal{O}(-j), \mathcal{O}(-i))$ and one gets the $\mathbb{Z}$-algebra $\bigoplus_{i,j}D_{ij}$ with multiplication induced by composition.  In fact, given any collection of objects in a category indexed by a set $I$, one can form the associated $I$-algebra in a similar way.

Not only can one recover the theory of $\mathbb{Z}$-graded algebras from that of $\mathbb{Z}$-algebras in the sense that every $\mathbb{Z}$-graded algebra is graded Morita equivalent to a $\mathbb{Z}$-algebra \cite{sierra}, \cite{quadrics}, but also various results from graded ring theory have much simpler conceptual proofs when viewed from the $\mathbb{Z}$-algebra perspective (see, for example \cite[Theorem 1.2]{sierra}).  As another example of the utility of $\mathbb{Z}$-algebras, M. Van den Bergh discovered, via deformation theoretic arguments, that there should be more noncommutative quadric surfaces than had been found as noncommutative spaces arising from cubic AS-regular algebras of dimension 3, and realized these missing spaces as having $\mathbb{Z}$-algebras coordinate rings \cite{quadrics}.  Van den Bergh also discovered a general notion of noncommutative $\mathbb{P}^{1}$-bundles over commutative schemes using $\mathbb{Z}$-algebras that generalized the initial notion studied by D. Patrick utilizing $\mathbb{Z}$-graded algebras \cite{patrick}.

Although the $\mathbb{Z}$-algebra approach to noncommutative projective geometry may be the more natural one as discussed above, for historical reasons it has not been pursued as vigorously as the $\mathbb{Z}$-graded approach.  The goal of this paper is to develop enough basic homological machinery of modules over $\mathbb{Z}$-algebras to prove a $\mathbb{Z}$-algebra version of local duality \cite[Theorem 5.1(2)]{dualizing}.  As an application of local duality, we explore the relationship between several notions of regularity studied in \cite{minamoto} in the $\mathbb{Z}$-graded context, AS-regularity and ASF-regularity.  In particular we show, in Theorem \ref{theorem.main2}, that these conditions are equivalent under a hypothesis which, in the $\mathbb{Z}$-graded context, follows from the existence of a balanced dualizing complex.  In the $\mathbb{Z}$-graded case, the equivalence of these conditions follows from \cite[Theorem 3.12]{minamoto} and \cite[Theorem 2.19]{moriueyama}.  In addition, we employ local duality to prove a $\mathbb{Z}$-algebra version of \cite[Theorem 3.3]{jorgensen}, which gives a duality between the bounded derived categories ${\sf D}^{b}({\sf coh }A)$ and ${\sf D}^{b}({\sf coh }A^{op})$.  As is stated and substantiated in \cite{jorgensen}, this duality is fundamental to many of the applications of local duality in both the commutative and noncommutative worlds.

One of the themes of projective geometry is to characterize a space via its homological properties.  Although there are currently few such results in noncommutative projective geometry, the theory would benefit from further development in this direction.  In \cite{abstractpn} we use both local duality and part of our comparison of regularity conditions (Corollary \ref{cor.af}) to obtain a $\mathbb{Z}$-algebra version of the recognition theorem of Mori and Ueyama \cite{moriueyama} characterizing those abelian categories which are noncommutative spaces having a $\mathbb{Z}$-graded AS-regular coordinate ring.

We now give a brief description of the contents of this paper.  In Section \ref{sec.iindex}, we recall the definition and basic properties of $I$-algebras and their opposites. We then define, in Section \ref{section.basics}, the main objects of study in this paper, connected $\mathbb{Z}$-algebras.  It turns out that, as in the $\mathbb{Z}$-graded case, the Ext-finite condition introduced in \cite{dualizing} plays a key role in studying and using the torsion functor in the context of local duality.  Next, in Sections \ref{section.tensor} and \ref{section.shom}, we introduce internal tensor functors and their left derived functors and internal hom functors and their right derived functors, and prove enough of the standard results about them to establish local duality and study regularity.  The reason these functors must be redefined in the $\mathbb{Z}$-algebra setting stems from the fact that, unlike in the $\mathbb{Z}$-graded case, bimodules over $\mathbb{Z}$-algebras are indexed by {\it pairs} of integers.  Although these sections contain material similar to results appearing in \cite{abstract} and \cite{chan}, they develop the material in greater generality in order to accommodate the applications in this paper.  In Section \ref{section.ld}, we define and study derived functors of hom and tensor, and then prove local duality.  We follow the approach to this material from \cite{amnon}, making some minor changes for the $\mathbb{Z}$-algebra setting.  Finally, in Section \ref{section.regularity}, we use local duality to prove Theorem \ref{theorem.main}, our version of \cite[Theorem 3.3]{jorgensen} mentioned above.  We conclude the paper by describing $\mathbb{Z}$-algebra versions of two notions of regularity studied in \cite{minamoto}, and showing that these conditions are equivalent (Theorem \ref{theorem.main2}).

\section{$I$-algebras}  \label{sec.iindex}

Let $I$ be a set.  In this section, we recall (from \cite[Section 2]{quadrics}) the notions of $I$-indexed algebras and $I$-indexed analogues of various ring theoretic concepts.  We recall that an {\em $I$-indexed algebra} $A$ is a pre-additive category whose objects are indexed by $I$ and denoted $\{\mathcal{O}_{i}\}_{i \in I}$, and whose morphisms are denoted $A_{ij} := \operatorname{Hom}(\mathcal{O}_{j}, \mathcal{O}_{i})$. If the category is in fact $k$-linear, which we assume from now on, then we say $A$ is an {\em $I$-indexed $k$-algebra} or, abusing terminology, an {\em $I$-algebra}.   We will often further abuse terminology and call the ring $A = \bigoplus_{i,j}A_{ij}$ with multiplication given by composition an $I$-algebra. We let $e_{i}$ denote the identity in $A_{ii}$.

Let $A$ be an $I$-indexed algebra. A {\em (graded) right $A$-module} is a graded abelian group $M = \bigoplus_{i \in I} M_i$ with multiplication maps $M_i \times A_{ij} \ra M_j$ satisfying the usual module axioms (see \cite{quadrics} for more details). We let ${\sf Gr }A$ denote the category of graded right $A$-modules and we let $A-{\sf Gr}$ denote the category of graded left $A$-modules.  We call a graded right $A$-module {\it free} if it is a direct sum of modules of the form $e_{i}A$ for possibly different $i \in I$, and {\it free and finitely generated} is this direct sum is finite.  By (see \cite[Section 2]{quadrics}), the category ${\sf Gr }A$ is a Grothendieck category so has enough injectives, and the objects $e_{i}A$ are projective, so that ${\sf Gr }A$ has enough projectives.

As for rings, one way to see the symmetry between left and right modules is to introduce the {\em opposite algebra $A^{op}$} which is just the opposite category.  Specifically, $A^{op}$ is the $I$-algebra with $(A^{op})_{ij} = A_{ji}$, and with multiplication of $f \in (A^{op})_{ij}$ and $g \in (A^{op})_{jk}$ defined as $f\cdot g:=gf \in A_{ki}=(A^{op})_{ik}$.

Throughout this paper, we let $K$ denote the $\mathbb{Z}$-algebra with
$$
K_{ij} := \begin{cases} k & \mbox{if $i=j$} \\ 0 & \mbox{otherwise}\end{cases}
$$
and with multiplication induced by that of $k$.

\begin{definition} \label{def.B}
Let $A$ and $B$ be $I$-algebras.  We let ${\sf Bimod }(A-B)$ denote the following category:  an object of ${\sf Bimod }(A-B)$ is a set $M:= \bigoplus_{i,j \in I}M_{ij}$ where $M_{ij}$ is an abelian group for all $i,j \in I$, together with group homomorphisms $\mu_{ijk}:M_{ij} \otimes_{k} B_{jk} \rightarrow M_{ik}$ and $\psi_{ijk}: A_{ij} \otimes_{k} M_{jk} \rightarrow M_{ik}$ for all $i,j,k \in I$ making $M$ an $A$-$B$ bimodule.  Morphisms of objects in ${\sf Bimod }(A-B)$ are defined in the obvious way.
\end{definition}

Let $A$ be an $I$-algebra and let $B$ be an $H$-algebra. We define an $I \times H$-algebra $A \otimes_k B$ by
$$(A \otimes_k B)_{(i,h)(i',h')} := A_{ii'} \otimes_k B_{hh'} .$$
If $A$ is an $I$-algebra, then the {\em enveloping algebra} $A^{op} \otimes_k A$ is an $I^2$-algebra.

The next result collects various equivalences between graded module categories.  To state it we introduce one more construction:  if $A$ is a $\mathbb{Z}$-algebra, we let $\widetilde{A^{op}}$ be the $\mathbb{Z}$-algebra defined by setting $(\widetilde{A^{op}})_{ij} := A_{-j,-i}$, with multiplication defined as follows: if $f \in (\widetilde{A^{op}})_{ij}$ and $g \in (\widetilde{A^{op}})_{jk}$, we let $f \cdot g := gf \in A_{-k,-i} = (\widetilde{A^{op}})_{ik}$.  The notation $\widetilde{A^{op}}$ will be justified in Proposition \ref{prop.equiv}.  We will also employ the following notation: if ${\sf C}$ and ${\sf D}$ are categories, we write ${\sf C} \equiv {\sf D}$ to mean there is an equivalence of categories ${\sf C} \rightarrow {\sf D}$.

\begin{prop} \label{prop.equiv}
Let $A$ and $B$ be $I$-algebras.  There are equivalences
\begin{enumerate}
\item{} ${\sf Bimod }(A-B) \equiv {\sf Gr }A^{op}\otimes_k B$,

\item{} $A^{op}-{\sf Gr} \equiv {\sf Gr }A$,

\item{} ${\sf Gr }A^{op} \equiv A-{\sf Gr}$,
\end{enumerate}
Furthermore, if $A$ is a $\mathbb{Z}$-algebra, then $A^{op}$ is left and right graded Morita equivalent to $\widetilde{A^{op}}$.
\end{prop}

\begin{proof}

The proof is routine but we include part of the proof of the last result.  We define a functor $F:{\sf Gr }A \rightarrow \widetilde{A^{op}}-{\sf Gr}$ as follows: we let $F(M)_{i} := M_{-i}$ and if $a \in (\widetilde{A^{op}})_{ji}= A_{-i,-j}$ and $m \in F(M)_{i}$, we define $a \cdot m := ma \in M_{-j}=F(M)_{j}$.  If $f:M \rightarrow N$ is a morphism in ${\sf Gr }A$, we define $F(f)_{i}=f_{-i}$.  It is now routine to check that $F$ has an inverse functor so is an equivalence.  Therefore, the fact that $A^{op}$ and $\widetilde{A^{op}}$ are left graded Morita equivalent follows from part (2).
\end{proof}

We now describe notation we will invoke in the sequel.  Let $A, B$ be $I$-algebras.  For each $h \in I$, there is a natural restriction functor $\operatorname{res}_h:{\sf Bimod}\, (A-B) \ra {\sf Gr}\, B$ defined by $M \mapsto e_h M$ and similarly for left modules.  Furthermore, for the proof of Lemma \ref{lemma.lrinjective} and the statement of Lemma \ref{lemma.bproj} below, we will use the following notation:  if $M\in A-{\sf Gr}$ and $N\in {\sf Gr} B$, then we define $M\otimes _kN\in {\sf Bimod} (A-B)$ by $(M\otimes _kN)_{ij}=M_i\otimes _kN_j$.

\begin{lemma} \label{lemma.lrinjective}
Let $h \in I$.  If $J$ is an injective $A-B$-bimodule, then $Je_{h}$ is an injective left $A$-module and $e_{h}J$ is an injective right $B$-module.
\end{lemma}

\begin{proof}
There is an exact left adjoint to $\operatorname{res}_{h}$, $Ae_h \otimes_k -$, which sends the $B$-module $P$ to the $A-B$-bimodule defined by $A e_h \otimes_k P$. The result follows from this.
\end{proof}

There are also the usual restrictions induced by morphisms of $I$-algebras:  We let ${}_{A}\operatorname{Res}:{\sf Bimod }(A-B) \rightarrow {\sf Bimod }(A-K)$, $\operatorname{Res}_{B}:{\sf Bimod }(A-B) \rightarrow {\sf Bimod }(K-B)$ and $\operatorname{Res}_{k}:{\sf Bimod }(A-B) \rightarrow {\sf Bimod }(K-K)$ denote the obvious restriction functors.  We call an object $M \in {\sf Bimod }(A-B)$ a {\it $B$-injective} if $\operatorname{Res}_{B} M$ is injective, and we define {\it $B$-projective} similarly.
\begin{lemma} \label{lemma.bproj}
If $M \in A-{\sf Gr}$, then the object $M \otimes_{k} e_{j}B$ of ${\sf Bimod }(A-B)$ is a $B$-projective for all $j \in I$.  Therefore, $\{Ae_{i} \otimes_{k} e_{j}B\}_{i,j \in I}$ is a set of projective generators for ${\sf Bimod }(A-B)$ consisting of $B$-projectives.
\end{lemma}

\begin{proof}
If we let ${}_{i}M$ denote the left $K$-module with
$$
{}_{i}M_{j}=\begin{cases} M_{i} & \mbox{if }j=i \\ 0 & \mbox{otherwise,}
\end{cases}
$$
then there is a canonical $K-B$-bimodule isomorphism
$$
M\otimes_{k}e_{j}B \rightarrow \underset{i}{\bigoplus}  { }_{i}M \otimes_{k} e_{j}B.
$$
Therefore, to prove the first result, it suffices to show $Ke_{i} \otimes_{k} e_{j}B$ is a projective $K-B$-bimodule.  To prove this, suppose $N$ is a $K-B$-bimodule.  Then the map
$$
\operatorname{Hom}_{{\sf Bimod }(K-B)}(Ke_{i} \otimes_{k} e_{j}B, N) \rightarrow \operatorname{Hom}_{{\sf Gr }B}(e_{j}B,e_{i}N) \cong N_{ij}
$$
with first composite sending $f$ to $f(e_{i}\otimes -)$ is an isomorphism, establishing the first result.

To establish the second result, we first show that $\{Ae_{i} \otimes_{k} e_{j}B\}_{i,j \in I}$ is a set of generators for ${\sf Bimod }(A-B)$.  Let $e_j'$ denote the identity in $B_{jj}$ to avoid the potential confusion.  Since $e_i^{op}\otimes e_j'$ is the identity in $(A^{op}\otimes _kB)_{(i, j)(i, j)}=A_{ii}^{op}\otimes _kB_{jj}$, and since $Ae_{i} \otimes_{k} e_{j}'B$ is sent to $(e_i^{op}\otimes e_j')(A^{op}\otimes _kB)$ under the equivalence functor
$$
{\sf Bimod} (A-B)\to {\sf Gr}A^{op}\otimes _kB
$$
from Proposition \ref{prop.equiv}, we conclude that $\{Ae_{i} \otimes_{k} e_{j}'B\}_{i,j \in I}$ is a set of generators for ${\sf Bimod }(A-B)$.

The second part will now follow from the first part if we can show that $Ae_{i} \otimes_{k} e_{j}B$ is a projective $A-B$-bimodule. This follows since the map
$$
\operatorname{Hom}_{{\sf Bimod }(A-B)}(Ae_{i} \otimes_{k} e_{j}B,N) \rightarrow N_{ij}
$$
defined by sending $f$ to $f_{ij}(e_{i}\otimes e_{j})$ is an isomorphism, as the reader can check.
\end{proof}

\section{Connected $\mathbb{Z}$-algebras} \label{section.basics}
We begin this section by introducing some terminology.  Following \cite[Section 2]{quadrics}, we say that a $\mathbb{Z}$-algebra $A$ is {\it positively graded} if $A_{ij}=0$ for all $i>j$ and that $A$ is {\it negatively graded} if $A_{ij}=0$ for all $i<j$. For $n$ a nonnegative integer, we let $A_{\geq n}$ denote the subobject of $A$ in ${\sf Bimod }(A-A)$ given by $\bigoplus_{j-i \geq n} A_{ij}$.

The terminology ``Ext-finite" in the following definition is due to \cite[Section 4]{dualizing}.
\begin{definition} \label{def.3.1}
Let $A$ be a positively graded $\mathbb{Z}$-algebra.  Then
\begin{itemize}
\item{} $A$ is {\it connected} if for all $i$, $A_{ii}$ is a division ring over $k$,

\item{} if $M$ is in ${\sf Gr }A$ and $A$ is connected, an exact sequence
$$
\cdots F_{2} \overset{\psi_{2}}{\rightarrow} F_{1} \overset{\psi_{1}}{\rightarrow} F_{0} \rightarrow M \rightarrow 0
$$
is a {\it minimal free resolution of $M$} if for all $i \geq 1$, $\operatorname{im }\psi_{i} \subset F_{i-1}A_{\geq 1}$, and $F_{i-1}$ free,

\item{} $A$ is {\it right Ext-finite} if $A$ is connected, and for all $i$, $e_{i}A/(e_{i}A)_{\geq i+1}$ has a minimal free resolution, each of whose terms is free and finitely generated,

\item{} $A$ is {\it left Ext-finite} if $A$ is connected, and for all $i$, $Ae_{i}/{(Ae_{i})}_{\leq i-1}$ has a minimal free resolution, each of whose terms is a finite direct sum of modules of the form $Ae_{l}$ for various $l$, and

\item{} $A$ is {\it Ext-finite} if $A$ is both right Ext-finite and left Ext-finite.
\end{itemize}
\end{definition}

We note that if $A$ is positively graded, then $A^{op}$ is negatively graded, and $\widetilde{A^{op}}$ is positively graded.  Furthermore, if $A$ is connected, then so is $\widetilde{A^{op}}$.

For the remainder of the paper, we let $\tilde{e}_{i}$ denote the unit in $(\widetilde{A^{op}})_{ii}$.
\begin{lemma} \label{lemma.extagain}
If $A$ is left (right) Ext-finite, then $\widetilde{A^{op}}$ is right (left) Ext-finite.
\end{lemma}

\begin{proof}
Suppose $A$ is left Ext-finite.  Under the equivalence in the last part of Proposition \ref{prop.equiv}, $\tilde{e}_{i}\widetilde{A^{op}}$ goes to $Ae_{-i}$, while $(\tilde{e}_{i}\widetilde{A^{op}})_{\geq i+1}$ goes to $(Ae_{-i})_{\leq -i-1}$.  Therefore, $\widetilde{A^{op}}$ is right Ext-finite.  The proof of the other assertion is similar.
\end{proof}

\begin{remark} \label{remark.ext} If $A$ is Ext-finite, it follows that $A_{ij}$ is finite dimensional over both $A_{ii}$ and $A_{jj}$.
\end{remark}

\begin{lemma} \label{lemma.resshape}
Suppose $A$ is right Ext-finite.  Then for all $i \in \mathbb{Z}$ and all $n \geq 1$, $e_{i}(A/A_{\geq n})$ has a resolution, each of whose terms is free and finitely generated.
\end{lemma}

\begin{proof}
The conclusion is true when $n=1$ by definition of right Ext-finiteness.  Suppose the result holds for some $n \geq 1$.  Since there is a canonical exact sequence in ${\sf Gr }A$
$$
0 \rightarrow e_{i}(A_{\geq n}/A_{\geq n+1}) \rightarrow e_{i}(A/A_{\geq n+1}) \rightarrow e_{i}(A/A_{\geq n}) \rightarrow 0
$$
the result follows from \cite[Lemma 2.2.8]{weibel}, from the fact that
$$
e_{i}(A_{\geq n}/A_{\geq n+1}) \cong A_{i,i+n} \otimes_{A_{i+n,i+n}} e_{i+n}(A/A_{\geq 1}).
$$
and from Remark \ref{remark.ext}.
\end{proof}

Suppose $A$ is a $\mathbb{Z}$-algebra.  A graded right $A$-module $M$ is {\it right bounded} if $M_{n}=0$ for all $n>>0$.  We let ${\sf Tors }A$ denote the full subcategory of ${\sf Gr }A$ consisting of modules whose elements $m$ have the property that the right $A$-module generated by $m$ is right bounded.

\begin{lemma} \label{lemma.serresub}
If $A$ is a right Ext-finite, connected $\mathbb{Z}$-algebra, then ${\sf Tors }A$ is a Serre subcategory of ${\sf Gr A}$, and there is a torsion functor $\tau: {\sf Gr }A \rightarrow {\sf Tors }A$ which sends a module to its largest torsion submodule.
\end{lemma}

\begin{proof}
If $N \in {\sf Tors }A$, it is elementary to check that any submodule or quotient module of $N$ is also in ${\sf Tors }A$.  Conversely, suppose we have a short-exact sequence
$$
0 \rightarrow M \rightarrow N \rightarrow P \rightarrow 0
$$
with $M, P \in {\sf Tors }A$.  Let $n \in N_{j}$.  Then there exists an $l \geq j$ minimal such that $n \cdot A_{jl'} \subset M$ for all $l' \geq l$.  We prove the result by induction on $l-j$.  If $l=j$, then $n \in M$ so that $n$ is a torsion element.

Now we consider the general case.  Since $A$ is right Ext-finite, there are finitely many generators $a_{1}, \ldots, a_{m}$ of $(e_{j}A)_{\geq j+1}$.  By the induction hypothesis, $na_{i}$ is torsion for all $i$.  Thus, if $a_{i} \in A_{jq_{i}}$, then for sufficiently large $p$, $na_{i}A_{q_{i}p}=0$ for all $i$.  Therefore, for sufficiently large $p$, if $x \in A_{jp}$, then $n \cdot x =0$.  Therefore, $n$ is torsion.
\end{proof}


\subsection{Graded coherence}
We now review the basic facts about coherence from \cite{polish} which we will need in the sequel.  For the rest of Section \ref{section.basics}, we let $A$ denote a connected $\mathbb{Z}$-algebra.  We warn the reader that our grading convention and assumptions on $A$ differ slightly from those appearing in \cite{polish}.

We say that $M \in {\sf Gr }A$ is {\it finitely generated} if there is a surjection $F \rightarrow M$ where $F$ is free and finitely generated.  We say $M$ is {\it coherent} if it is finitely generated and if for every homomorphism $f: F \rightarrow M$ with $F$ free and finitely generated, $\operatorname{ker }f$ is finitely generated.  We denote the full subcategory of ${\sf Gr }A$ consisting of coherent modules by ${\sf coh }A$.  By \cite[Proposition 1.1]{polish}, ${\sf coh }A$ is an abelian subcategory of ${\sf Gr }A$ closed under extensions.

We call $A$ {\it right-coherent} if the right modules $e_{j}A$ and $e_{j}A/(e_{j}A)_{\geq j+1}$ are coherent.  In the case that $A$ is a connected $\mathbb{Z}$-algebra, finitely generated in degree one, then the condition that the modules $e_{j}A$ are coherent implies that $A$ is coherent.  The definition of left-coherence is similar.

The following is \cite[Lemma 1.2]{polish}.

\begin{lemma} \label{lemma.res}
If $A$ is right coherent, then every $M \in {\sf coh }A$ has a resolution
$$
\cdots \rightarrow F^{2} \rightarrow F^{1} \rightarrow F^{0} \rightarrow M \rightarrow 0
$$
with $F^{i}$ free and finitely generated.
\end{lemma}

\section{Internal tensor functors and their left derived functors} \label{section.tensor}
In this section, we develop the basic properties of the tensor product of bimodules and their associated left-derived functors.

\subsection{Internal tensor.}  \label{section.tensordef} Suppose $I$ is a set and $A$, $B$ and $C$ are $I$-algebras.  If $M$ is a graded $A$-module and $N$ is an $A-B$-bimodule, we let
$$
M \intotimes N := \operatorname{cok} \bigl(\bigoplus_{l,m} M_l \otimes_{A_{l,l}} A_{l,m} \otimes_{A_{m,m}} e_{m}N \xrightarrow{\mu \otimes 1 - 1 \otimes \mu} \bigoplus_n M_n \otimes_{A_{n,n}} e_{n}N \bigr).
$$
It is right exact in each component being defined by cokernels.

Now, suppose $M$ is an $A-B$-bimodule and $N$ is a $B-C$-bimodule.  We let $M \intotimesb N$ denote the $A-C$-bimodule defined by
$$
(M \intotimesb N)_{ij} := (e_{i}M \intotimesb N)_{j}
$$
with right-structure that of the right $C$-module $\bigoplus_{i}(e_{i}M \intotimesb N)$ and with left multiplication by $A$ defined in the obvious way.  It is also right-exact in each component, as one can check.  In the sequel we will use, without comment, the fact that multiplication induces an isomorphism $e_{i}B \intotimesb N \cong e_{i}N$ and an isomorphism $M \intotimesb B \cong M$.

\begin{remark} \label{remark.tensor} Let $N$ be a $B-C$-bimodule such that, for all $h \in I$, $Ne_{h}$ is a projective left $C$-module.  Then, as usual, the functor $-\intotimesb N:{\sf Bimod }(A-B) \rightarrow {\sf Bimod }(A-C)$ is exact.  A similar assertion holds replacing the first input with the second.
\end{remark}


The proof of the following result is elementary.
\begin{lemma} \label{lemma.restensor}
Let $M$ be an $A-B$-bimodule and let $N$ be a $B-C$-bimodule.  Then $M \intotimesb {}_{B}\operatorname{Res }N = {}_{A}\operatorname{Res }(M \intotimesb N)$.
\end{lemma}

For the remainder of Section \ref{section.tensor}, we assume that $A$, $B$ and $C$ are connected $\mathbb{Z}$-algebras.  We define $A_{0}$ as the cokernel of the canonical inclusion $A_{\geq 1} \rightarrow A$.  We say $M \in {\sf Gr }A$ is {\it left-bounded by $d$} if $M_{i}=0$ for $i<d$.  We say $N \in A-{\sf Gr}$ is {\it right-bounded by $d$} if $N_{i}=0$ for all $i>d$.

The next result is Nakayama's Lemma.

\begin{lemma} \label{lemma.nak}
Suppose $M$ is left-bounded.  Then $M \intotimes A_{0}=0$ implies that $M=0$.  Therefore, if $N$ is left-bounded and $f:M \rightarrow N$ is a morphism with $f \intotimes A_{0}$ a surjection, then $f$ is a surjection.
\end{lemma}

\begin{proof}
Suppose $M$ is left-bounded by $d$ and $M_{d} \neq 0$.  By right exactness of $M\intotimes - $, we have an exact sequence
$$
M \intotimes A_{\geq 1} \rightarrow M \intotimes A \cong M \rightarrow M \intotimes A_{0} \rightarrow 0.
$$
Since the image of the left map is contained in degree greater than $d$, the hypothesis implies that $M_{d}=0$, a contradiction. Thus, $M=0$.

To prove the second part, we apply $ -\intotimes A_{0}$ to
$$
0 \rightarrow \operatorname{im} f \rightarrow N \rightarrow \operatorname{cok }f \rightarrow 0
$$
to obtain an exact sequence
$$
(\operatorname{im }f) \intotimes A_{0} \rightarrow N \intotimes A_{0} \rightarrow \operatorname{cok }f \intotimes A_{0} \rightarrow 0,
$$
and we have $(\operatorname{im}f) \intotimes A_{0} \cong \operatorname{im}(f \intotimes A_{0}) = N \intotimes A_{0}$ by hypothesis so that
$$
\operatorname{cok }f \intotimes A_{0} = 0.
$$
Since $N$ left-bounded, the first part implies that $\operatorname{cok }f=0$ as desired.
\end{proof}

\begin{proposition} \label{prop.min}
Let $M$ be left-bounded by $d$.  Then there exists a multiset of integers $J$ and a surjection of the form
$$
\psi: F(:= \bigoplus_{j \in J}e_{j}A) \rightarrow M
$$
such that $\operatorname{ker }\psi \subset FA_{\geq 1}$ and $j \in J$ implies that $j \geq d$.  It follows that $M$ has a minimal free resolution.
\end{proposition}

\begin{proof}
The second result follows from the first since $\operatorname{ker }\psi$ is left bounded.  We now prove the first result.  Let $U$ be a graded $k$-subspace of $M$ defined as follows: for $j<d$, $U_{j}=0$, $U_{d}=M_{d}$, and for $j>0$, we let $U_{d+j}$ be an $A_{d+j, d+j}$ complementary subspace of $(MA_{\geq 1})_{d+j}$.  By construction, we have $M=U \oplus MA_{\geq 1}$, and there exists a canonical map, $\phi$, from a module of the form $\bigoplus_{j \in J}e_{j}A$ to $M$ sending local units in degree $j$ bijectively to a right $A_{jj}$-basis of $U_{j}$.  By Lemma \ref{lemma.nak}, $\phi$ is a surjection, and by construction, $\operatorname{ker }\phi$ is contained in $FA_{\geq 1}$ as desired.
\end{proof}
Now the usual proof, using Lemma \ref{lemma.nak}, establishes the following result.
\begin{corollary} \label{cor.min}
If $M$ is left-bounded and $\operatorname{pd }M=n$, then every minimal free resolution of $M$ has length $n$.
\end{corollary}



\subsection{Internal Tor.}  \label{section.tor}  Let $M$ be an $A-B$-bimodule.  As $-\intotimes M:{\sf Gr }A \rightarrow {\sf Gr }B$ is right exact and ${\sf Gr }A$ has enough projectives, the left-derived functors of $-\intotimes M$ exist and can be computed using projective resolutions.  We denote the $i$th such $\uTor_{i}^{A}(-,M)$.  It follows immediately that $\uTor_{i}^{A}(F,M)=0$ if $i>0$ and $F$ is projective.  In addition, the proof that the bifunctor $\uTor_{i}^{A}(-,-)$ is balanced follows the usual proof \cite[Theorem 2.7.2]{weibel}.  For the readers convenience, we include the $\mathbb{Z}$-algebra version of \cite[Proposition 4.4.11]{weibel} below.

\begin{lemma} \label{lemma.preliminary}
Let $N$ be a left-bounded object of ${\sf Gr }A$, let $\operatorname{pd} N \neq 0$, and suppose
$$
0 \rightarrow L \rightarrow F \rightarrow N \rightarrow 0
$$
is the start of a minimal free resolution (which exists by Proposition \ref{prop.min}).  Then
\begin{enumerate}
\item{} $\uTor_{d+1}^{A}(N,A_{0}) \cong \uTor_{d}^{A}(L,A_{0})$ for $d \geq 1$, and

\item{} $\operatorname{pd }N \leq 1+ \operatorname{pd }L$.
\end{enumerate}

\end{lemma}

\begin{proof}
The first result follows from the long exact sequence induced by applying $- \intotimes A_{0}$ to the given short exact sequence.

If $\cdots \rightarrow F_{i} \rightarrow \cdots \rightarrow F_{1} \rightarrow F_{0} \rightarrow N \rightarrow 0$ is the continuation of the given short exact sequence to a minimal free resolution for $N$, then it can be truncated to a minimal free resolution
$$
\cdots \rightarrow F_{i} \rightarrow \cdots \rightarrow F_{1} \rightarrow L \rightarrow 0
$$
of $L$, whence the second result.
\end{proof}

\begin{proposition} \label{prop.4.9}
Suppose $N$ is left-bounded by $d$.  Then $\operatorname{pd }N \leq n$ if and only if $\uTor_{n+1}^{A}(N,A_{0})=0$.  In particular, $N$ is projective if and only if $N$ has the form $\bigoplus_{j\in J}e_{j}A$ for some multiset $J$ such that $j \in J$ implies $j \geq d$.
\end{proposition}

\begin{proof}
The forward direction of the first statement is immediate since $\uTor_{n+1}^{A}(N,A_{0})$ is computable using projective resolutions of $N$.

Conversely, we first suppose $n=0$, i.e. $\uTor^{A}_{1}(N,A_{0})=0$.  If we have the start of a minimal free resolution
\begin{equation} \label{eqn.usualses}
0 \rightarrow L \rightarrow F \rightarrow N \rightarrow 0
\end{equation}
then, by hypothesis, the induced sequence
$$
0 \rightarrow L\intotimes A_{0} \rightarrow F \intotimes A_{0} \rightarrow N \intotimes A_{0} \rightarrow 0
$$
is exact.  Minimality of (\ref{eqn.usualses}) implies $L \intotimes A_{0}=0$, so that, since $L$ is left-bounded, Lemma \ref{lemma.nak} implies that $L=0$.  Therefore $N$ has the form $\bigoplus_{j\in J}e_{j}A$ for some multiset $J$ such that $j \in J$ implies $j \geq d$, so $\operatorname{pd }N=0$.  This establishes the backwards direction of the first statement in case $n=0$ and, together with the forward direction of the first statement, proves the second statement.

Finally, we complete the proof of the first statement.  Suppose $n>0$.  The first part of Lemma \ref{lemma.preliminary} implies $\uTor^{A}_{n}(L,A_{0})=0$ so that, by induction, $\operatorname{pd }(L) \leq n-1$.  Then the second part of Lemma \ref{lemma.preliminary} implies that $\operatorname{pd }N \leq n$.
\end{proof}
The next result follows immediately.
\begin{corollary} \label{cor.pdn}
Suppose $N$ is a left-bounded object of ${\sf Gr }A$.  Then
$$
\operatorname{pd} N = \operatorname{sup }\{p|\uTor^{A}_{p}(N,A_{0}) \neq 0\}.
$$
\end{corollary}

We note that $A_{0}e_{j}$ is an $A-A$-bimodule, so $\uTor^{A}_{p}(N,A_{0}e_{j})$ exists.  Furthermore, since $-\intotimes Ae_j$ is an exact functor, we have
$$
\uTor^{A}_{p}(N,A_{0})e_{j} \cong \uTor^{A}_{p}(N,A_{0}e_{j}).
$$
It thus follows from Corollary \ref{cor.pdn} that
$$
\operatorname{pd }N = \operatorname{sup }\{p|\uTor^{A}_{p}(N,A_{0}e_{j}) \neq 0 \mbox{ for some } j \in \mathbb{N}\}.
$$

If $M$ is a $B-A$-bimodule, one can also consider the left-derived functors of $M \intotimes -:A-{\sf Gr} \rightarrow B-{\sf Gr}$.  The above results (Lemma \ref{lemma.nak}, Proposition \ref{prop.min}, Lemma \ref{lemma.preliminary}, Proposition \ref{prop.4.9} and Corollary \ref{cor.pdn}) each have versions for right-bounded graded left $A$-modules with virtually identical proofs.  In particular, we single out

\begin{prop} \label{prop.pdn2}
Suppose $M$ is a right-bounded object of $A-{\sf Gr}$.  Then
$$
\operatorname{pd} M = \operatorname{sup }\{p|\uTor^{A}_{p}(A_{0},M) \neq 0\}.
$$
\end{prop}

\begin{cor} \label{cor.sbalance}
Let $A$ be a connected $\mathbb{Z}$-algebra.  Then
$$
\operatorname{sup }\{\operatorname{pd }e_{i}A_{0}| i \in \mathbb{Z}\}=\operatorname{sup }\{\operatorname{pd }A_{0}e_{j}| j \in \mathbb{Z}\}.
$$
\end{cor}

\begin{proof}
By Corollary \ref{cor.pdn} and Proposition \ref{prop.pdn2},
\begin{eqnarray*}
\operatorname{sup }\{\operatorname{pd }e_{i}A_{0}| i \in \mathbb{Z}\} & = & \operatorname{sup }\{p|\uTor^{A}_{p}(e_{i}A_{0},A_{0}e_{j}) \neq 0 \mbox{ for some } i,j \in \mathbb{N}\} \\
& = & \operatorname{sup }\{\operatorname{pd }A_{0}e_{j}| j \in \mathbb{Z}\}.
\end{eqnarray*}
\end{proof}

\begin{proposition} \label{prop.gldim}
Let $A$ be a connected $\mathbb{Z}$-algebra such that
$$
\operatorname{sup }\{\operatorname{pd }A_{0}e_{j}| j \in \mathbb{Z}\} = n < \infty
$$
and let $N$ be a left-bounded object of ${\sf Gr }A$.  Then $\operatorname{pd }N \leq n$.
\end{proposition}

\begin{proof}
By the comment following Corollary \ref{cor.pdn}, it suffices to prove that
$$
\uTor^{A}_{p}(N,A_{0}e_{i}) =  0
$$
for all $i \in \mathbb{Z}$ and all $p \geq n+1$.  Since $N$ is left-bounded, it has a minimal free resolution by Proposition \ref{prop.min}.  Tensoring this resolution with $A_{0}$ and taking homology at the $p$th term yields $\uTor^{A}_{p}(N,A_{0})$.  Since $\uTor^{A}_{p}(-,-)$ is balanced, $\uTor^{A}_{p}(N,A_{0}e_{i})=\uTor^{A}_{p}(N,A_{0})e_{i}$ can also be computed by tensoring a minimal free resolution of $A_{0}e_{i}$ by $N$.  The result now follows from Corollary \ref{cor.min}.
\end{proof}

\section{Internal Hom functors and their right derived functors} \label{section.shom}
In Section \ref{section.inthom} and Section \ref{section.intext} we review the definition and basic properties of the internal Hom functors and their derived functors studied in \cite{abstract}.

\subsection{Internal Hom} \label{section.inthom}
We begin by defining the internal Hom functor.  Suppose $I$ is a set and $A$, $B$, $C$ and $D$ are $I$-algebras.  For $M$ an object in ${\sf Bimod }(A-B)$ and $N$ an object in ${\sf Gr }B$, we let
$$
\sHomb(e_{i}M,N)
$$
denote the right $A_{ii}$-module with underlying set  $\operatorname{Hom}_{{\sf Gr}B}(e_{i}M,N)$ and with $A_{ii}$-action induced by the left action of $A_{ii}$ on $e_{i}M$, and we let
$$
{\usHomb}(M,N)
$$
denote the object in ${\sf Gr }A$ with $i$th component $\sHomb(e_{i}M,N)$ and with multiplication induced by left-multiplication of $A$ on $M$.


Next, suppose $M$ is an $A-B$-bimodule and $N$ is a $C-B$-bimodule.  We define a $C-A$-bimodule $\usHomb(M,N)$, as follows: we let
$$
\usHomb(M,N)_{ij} := \usHomb(M,e_{i}N)_{j},
$$
we endow $\usHomb(M,N)$ with right $A$-module multiplication inherited from that of $\usHomb(M,e_{i}N)$ (where we consider $e_{i}N$ to be an object of ${\sf Gr }B$), and we let the left $C$-module multiplication be canonical.

\begin{remark} \label{remark.hom} We note that $\usHomb(-,-)$ is left exact in each variable, and for each injective $J$ in ${\sf Bimod }(C-B)$, $\usHomb(-,J)$ is exact.
\end{remark}


The following result is elementary but will be needed later.
\begin{lemma} \label{lemma.restriction}
Let $M$ be an $A-B$-bimodule and let $N$ be a $C-B$-bimodule.  Then there is an equality
$$
\operatorname{Res}_{k}(\usHomb(M,N))=\usHomb(\operatorname{Res}_{B}M,\operatorname{Res}_{B}N).
$$
\end{lemma}

As expected, one has adjointness of internal Hom and tensor:

\begin{prop} \label{prop.intadjoint}
Let $M$ be an $A-B$-bimodule, let $N$ be a $B-C$-bimodule, and let $P$ be a $D-C$-bimodule.  Then
\begin{enumerate}
\item{}
for each $i,j \in I$, there is a canonical isomorphism of $D_{ii}-A_{jj}$-bimodules natural in all inputs
$$
\psi_{ij}: \sHomc((e_{j}M)\intotimesb N, e_{i}P) \overset{\cong}{\longrightarrow} \sHomb(e_{j}M, \usHomc(N,e_{i}P)),
$$

\item{}
there is a canonical isomorphism of $D-A$-bimodules natural in all inputs
$$
\Psi: \usHomc(M \intotimesb N,P) \overset{\cong}{\longrightarrow} \usHomb(M, \usHomc(N,P)).
$$
\end{enumerate}
\end{prop}

\begin{proof}
The first result is straightforward and follows the usual proof in the graded case.  For the second result, let $\Psi_{ij}:= \psi_{ij}$.  Then the fact that $\Psi$ is natural and a bimodule map follows from naturality in part (1).
\end{proof}

\subsection{The duality functor} \label{section.duality}
In this section, we define the duality functor and describe its basic properties.  Once again, we let $I$ be a set and we let $A$, $B$ and $C$ denote $I$-algebras.  If $M \in {\sf Bimod }(A-B)$, then we first observe that the $K-A$-bimodule $\usHomk({}_{A}\operatorname{Res }M,K)$ has a left $B$-module structure induced from the right structure of $M$, and we let
$$
D: {\sf Bimod }(A-B) \rightarrow {\sf Bimod }(B-A)
$$
denote the functor
$$
D(M) = \usHomk({}_{A}\operatorname{Res }M,K)
$$
with the left $B$-module structure defined above.  The functor $D$ is exact.

Furthermore, if $M$ is an object of ${\sf Gr }B$, we may consider it as an object $\overline{M}$ of ${\sf Bimod }(K-B)$ by defining
$$
\overline{M}_{ij} := \begin{cases} M_{j} & \mbox{ if $i=0$}, \\ 0 & \mbox{otherwise.} \end{cases}
$$
In this case, we let $D(M) := D(\overline{M})$.  A similar comment applies to $M$ in $B-{\sf Gr}$.

We omit the straightforward proof of the next lemma, which relates $D$ to the usual duality functor.
\begin{lemma} \label{lemma.usual}
Consider the functor $E:{\sf Bimod }(A-B) \rightarrow {\sf Bimod }(B-A)$ defined by letting $E(M)_{ij} := \operatorname{Hom}_{k}((M_{ji})_{k},k)$ and with bimodule structure induced by that of $M$.  Then $D \cong E$.
\end{lemma}

\begin{prop} \label{prop.duality}
Let $M$ be an $A-B$-bimodule and let $N$ be an $B-C$-bimodule.  Then there are left $C$-module structures on $\usHomb(M,\usHomk({}_{B}\operatorname{Res}N,K))$ and $\usHomk({}_{K}\operatorname{Res}(M \intotimesb N),K)$ such that the $K-A$-bimodule isomorphism
\begin{eqnarray*}
\usHomb(M,\usHomk({}_{B}\operatorname{Res}N,K)) & \cong & \usHomk(M \intotimesb {}_{B}\operatorname{Res }N,K) \\
& = & \usHomk({}_{A}\operatorname{Res}(M \intotimesb N),K)
\end{eqnarray*}
with first composite from Proposition \ref{prop.intadjoint} and with second composite from Lemma \ref{lemma.restensor}, is compatible with left $C$-module multiplication.  Therefore there is a natural isomorphism
$$
\usHomb(M, D(N)) \cong D(M \intotimesb N).
$$
\end{prop}

\begin{proof}
We endow $\usHomb(M,\usHomk({}_{B}\operatorname{Res}N,K))$ with left $C$-module structure coming from the left structure on $\usHomk({}_{B}\operatorname{Res}N,K)$ defined above and similarly, we give $\usHomk({}_{A}\operatorname{Res}(M \intotimesb N),K)$ the left module structure as in the definition of the duality functor.

Next, let $c_{li} \in C_{li}$.  It suffices to show that the diagram
$$
\begin{CD}
\operatorname{Hom}_{K}((e_{j}M)\intotimesb {}_{B}\operatorname{Res }N,e_{i}K) & \longrightarrow & \operatorname{Hom}_{K}(e_{j}M,\usHomk({}_{B}\operatorname{Res }N,e_{i}K)) \\
@VVV @VVV \\
\operatorname{Hom}_{K}((e_{j}M)\intotimesb {}_{B}\operatorname{Res }N,e_{l}K) & \longrightarrow & \operatorname{Hom}_{K}(e_{j}M,\usHomk({}_{B}\operatorname{Res }N,e_{l}K))
\end{CD}
$$
with horizontals from adjointness and with verticals induced by left multiplication by $c_{li}$, commutes for all $j \in I$.  This can be checked explicitly and we omit the routine computation.
\end{proof}





\subsection{Internal Ext} \label{section.intext}
Let $I$ be a set, let $A$ and $B$ denote $I$-algebras, and let $M \in {\sf Bimod }(A-B)$.  In this section we study the right derived functors of $\usHomb(M,-)$ and $\sHomb(e_{j}M,-)$ considered as functors from ${\sf Gr }B$ to ${\sf Gr }B$.  The fact that $\usHomb(M,-)$ and $\sHomb(e_{j}M,-)$ have right derived functors follows from Remark \ref{remark.hom}.  We denote them by $\usExtb^{i}(M,-)$ and $\sExtb^{i}(e_{j}M,-)$.
We will invoke the following notation in this section: if $R$ is a ring, we let ${\sf Mod }R$ denote the category of right $R$-modules. If $S$ is a ring and $F$ is an $R-S$-bimodule, we let $F^{*}$ denote the right dual of $F$.  We note that since taking the $j$th degree part of an object of ${\sf Gr }B$ is an exact functor from ${\sf Gr }B$ to ${\sf Mod }B_{jj}$, we have
$$
(\usExtb^{i}(M,N))_{j} \cong \sExtb^{i}(e_{j}M,N)
$$
for all $N$ in ${\sf Gr }B$.  We further note that the sequence $(\sExtb^{i}(M,-))_{i \geq 0}$ forms a universal $\delta$-functor by \cite[III, Corollary 1.4]{hartshorne0}.

For the remainder of this section, we specialize to the case that $A$ is a connected $\mathbb{Z}$-algebra and $M$ is an $A-A$-bimodule.  For the readers convenience, we recall a result from \cite{abstract}, (Lemma \ref{lemma.extcommute} below), which will be needed in Section \ref{section.regularity}.  In order to do this, we introduce some notation.  Suppose $F$ is an $A_{jj}-A_{ii}$-bimodule of finite dimension on either side, and let $F \otimes_{A_{ii}} e_{i}M$ denote the object of ${\sf Bimod }(A-A)$ such that
$$
(F \otimes_{A_{ii}} e_{i}M)_{lm} = \begin{cases} F \otimes_{A_{ii}}M_{im} & \mbox{if $l=j$} \\ 0 & \mbox{otherwise,} \end{cases}
$$
endowed with the obvious bimodule structure.  We recall, from \cite[Lemma 6.5]{abstract}, that the proof of the following lemma comes from uniqueness of universal $\delta$-functors as in \cite[III, Remark 1.2.1]{hartshorne0}.
\begin{lemma}  \label{lemma.extcommute}
Suppose $F$ is an $A_{jj}-A_{ii}$-bimodule of finite dimension on either side, and let $F \otimes_{A_{ii}} e_{i}M$ denote the object of ${\sf Bimod }(A-A)$ defined above.  Then there is a natural isomorphism of ${\sf Mod }A_{jj}$-valued functors
$$
\sExt^{j}(F \otimes_{A_{ii}} e_{i}M,-) \cong \sExt^{j}(e_{i}M,-) \otimes_{A_{jj}} F^{*}.
$$
\end{lemma}

\begin{corollary} \label{cor.extstar}
Suppose $A$ is Ext-finite.  Then
$$
\usExt^{i}(A_{\geq n}/A_{\geq n+1},-)_j \cong \sExt^{i}(e_{j+n}A_0,-) \otimes_{A_{j+n, j+n}} A_{j, j+n}^{*}.
$$
\end{corollary}

\begin{proof}
We have isomorphisms
\begin{align*}
\usExt^{i}(A_{\geq n}/A_{\geq n+1},-)_j
& \cong \sExt^{i}(e_j(A_{\geq n}/A_{\geq n+1}),-)  \\
& \cong \sExt^{i}(A_{j, j+n}\otimes _{A_{j+n, j+n}}e_{j+n}A_0,-)  \\
& \cong \sExt^{i}(e_{j+n}A_0,-) \otimes_{A_{j+n, j+n}} A_{j, j+n}^{*}
\end{align*}
where the last isomorphism is from Lemma \ref{lemma.extcommute}, which applies since $A_{j, j+n}$ is finite dimensional over $A_{jj}$ and $A_{j+n, j+n}$ by our assumption of Ext-finiteness and Remark \ref{remark.ext}.
\end{proof}


Given the description of ${\sf Tors }A$ from Lemma \ref{lemma.serresub}, the following is elementary to check.
\begin{lemma} \label{lemma.torsfunct}
Suppose $A$ is right Ext-finite.  Then there is an isomorphism of functors $\tau(-) \cong \underset{n \to \infty}{\lim} \usHom(A/A_{\geq n}, -)$.
\end{lemma}

\begin{lemma} \label{lemma.dlims}
Suppose $A$ is right Ext-finite.  For $i \geq 0$, the functor $\operatorname{R}^{i}\tau$ commutes with direct limits.
\end{lemma}

\begin{proof}
By Lemma \ref{lemma.resshape}, $e_{j}A/e_{j}(A_{\geq n})$ has a resolution whose terms are free and finitely generated.  By \cite[Proposition 8.2, p. 809]{lang}, $\sExt^{i}(e_{j}A/e_{j}(A_{\geq n}),-)$ can be computed as the $i$th cohomology of $\sHom(-,-)$ applied to this resolution.  Therefore, the functor
$$
\dlim \sExt^{i}(e_{j}A/e_{j}(A_{\geq n}),-)
$$
commutes with direct limits, as desired.
\end{proof}

\begin{lemma} \label{lemma.torsvanish}
Suppose $A$ is right Ext-finite.  If $T \in {\sf Gr }A$ is a torsion module, then $\mbox{R}^{i}\tau(T)=0$ for $i>0$.
\end{lemma}

\begin{proof}
By Lemma \ref{lemma.dlims}, we may assume, without loss of generality, that $T$ is right-bounded by degree $r$.  The result now follows from \cite[Lemma 6.4(3)]{abstract}.
\end{proof}

\section{Local Duality} \label{section.ld}
Our goal in this section is to prove a version of local duality for connected right Ext-finite $\mathbb{Z}$-algebras.  As many of the proofs in this section are similar to their $\mathbb{Z}$-graded counterparts, we will rely heavily on proofs already in the literature, while carefully describing the differences that arise.

Throughout this section, unless otherwise stated, we let $A$, $B$, $C$ and $D$ be $\mathbb{Z}$-algebras.  If ${\sf C}$ is an abelian category, we let ${\sf K}({\sf C})$ denote the homotopic category of ${\sf C}$, we let ${\sf K}^{+}({\sf C})$ (resp. ${\sf K}^{-}({\sf C})$) denote the full subcategory of ${\sf K}({\sf C})$ consisting of bounded below (resp. bounded above) complexes. We let ${\sf D}({\sf C})$ denote the derived category of ${\sf C}$, and we let ${\sf D}^{+}({\sf C})$ (resp. ${\sf D}^{-}({\sf C})$) denote the full subcategory of ${\sf D}({\sf C})$ consisting of bounded below (resp. bounded above) complexes.  Finally, we let ${\sf D}^{b}({\sf C})$ denote the full subcategory of ${\sf D}({\sf C})$ consisting of bounded complexes.

Suppose $M^{\bullet}$ is a complex of $A-B$-bimodules, $N^{\bullet}$ is a complex of $B-C$-bimodules, and $P^{\bullet}$ is a complex of $D-C$-bimodules.  We define $\operatorname{Tot}(M^{\bullet} \intotimesb N^{\bullet})$ as in \cite[Definition 12.23.3]{stacks} or, equivalently \cite[p. 51]{amnon} and $\usHomc^{\bullet}(N^{\bullet},P^{\bullet})$ as in \cite[Section 15.67]{stacks} or, equivalently \cite[p. 49]{amnon}.  In light of Proposition \ref{prop.intadjoint}(2), the proof of the following proposition is completely analogous to the proof of \cite[Lemma 15.67.1]{stacks} and we omit the details.

\begin{proposition} \label{prop.adjointcomplex}
There is a natural isomorphism of complexes
$$
\usHomc^{\bullet}(\operatorname{Tot}(M^{\bullet} \intotimesb N^{\bullet}),P^{\bullet}) \overset{\cong}{\longrightarrow} \usHomb^{\bullet}(M^{\bullet}, \usHomc^{\bullet}(N^{\bullet},P^{\bullet})).
$$
\end{proposition}

\begin{cor} \label{cor.D}
There is a natural isomorphism of complexes of $C-A$-bimodules
$$
\usHomb^{\bullet}(M^{\bullet},D(N^{\bullet})) \cong D(\operatorname{Tot}(M^{\bullet} \intotimes N^{\bullet})).
$$
\end{cor}

\begin{proof}
Since ${}_{A}\operatorname{Res } \operatorname{Tot }(M^{\bullet} \intotimesb N^{\bullet}) = \operatorname{Tot }(M^{\bullet} \intotimesb {}_{B}\operatorname{Res }N^{\bullet})$, it suffices to show that the adjoint isomorphism from Proposition \ref{prop.adjointcomplex}
$$
\usHomb^{\bullet}(M^{\bullet}, \usHomk^{\bullet}({}_{B}\operatorname{Res }N^{\bullet}, K)) \cong \usHomk^{\bullet}(\operatorname{Tot}(M^{\bullet} \intotimesb {}_{B} \operatorname{Res }N^{\bullet}), K)
$$
is an isomorphism of complexes of left $C$-modules.  This follows from Proposition \ref{prop.duality} in light of the explicit form of the isomorphism from Proposition \ref{prop.adjointcomplex}.
\end{proof}

Next, we work towards definitions of derived functors of internal tensor and hom functors.  The proof of the following is analogous to the proof of \cite[Lemma 21]{risingsea} and is omitted.
\begin{lemma} \label{lemma.homexist0}
The functor
$$
\usHomb(-,-):({{\sf Bimod }(A-B)})^{op} \times {\sf Bimod }(C-B) \rightarrow {\sf Bimod}(C-A)
$$
descends to a a bi-$\delta$-functor
$$
\usHomb^{\bullet}(-,-):{\sf K}({{\sf Bimod }(A-B)})^{op} \times {\sf K}({\sf Bimod }(C-B)) \rightarrow {\sf K}({\sf Bimod}(C-A)).
$$
\end{lemma}
We now adapt \cite[Theorem 2.2]{amnon} to our context.  To do so, we note the next result, which follows from Lemma \ref{lemma.restriction}.
\begin{lemma} \label{lemma.restriction2}
There is an equality
$$
\operatorname{Res}_{k}(\usHomb^{\bullet}(M^{\bullet},N^{\bullet}))=
\usHomb^{\bullet}(\operatorname{Res}_{B}M^{\bullet},\operatorname{Res}_{B}N^{\bullet}).
$$
\end{lemma}
As in the proof of \cite[Theorem 2.2]{amnon}, we will need the following version of \cite[I, Lemma 6.2]{hartshorne}.
\begin{lemma} \label{lemma.hartshorne}
Let ${\sf L}$ denote the full subcategory of ${\sf K}^{+}({\sf Bimod }(C-B))$ consisting of complexes which are isomorphic, in ${\sf K}^{+}({\sf Bimod }(C-B))$, to complexes of $B$-injectives.  Suppose $M^{\bullet}$ is in ${\sf K}({\sf Bimod }(A-B))$ and $N^{\bullet} \in {\sf L}$.  If either $M^{\bullet}$ or $N^{\bullet}$ is acyclic, then $\usHomb^{\bullet}(M^{\bullet},N^{\bullet})$ is acyclic.
\end{lemma}

\begin{proof}
By Lemma \ref{lemma.restriction2}, it suffices to show that, under the hypotheses,
$$
\usHomb^{\bullet}(\operatorname{Res}_{B}M^{\bullet},\operatorname{Res}_{B}N^{\bullet})
$$
is acyclic.  This is a complex of bigraded bimodules, which is acyclic if and only if it is acyclic in each bidegree.  Therefore, the result follows from Lemma \ref{lemma.lrinjective} and \cite[I, Lemma 6.2]{hartshorne}.
\end{proof}

Our adaptation of \cite[Theorem 2.2]{amnon} follows.
\begin{proposition} \label{prop.homexist}
\begin{enumerate}
\item{}
The functor $\usHomb^{\bullet}(-,-)$ from Lemma \ref{lemma.homexist0} has a derived functor
$$
\operatorname{R}\usHomb^{\bullet}(-,-):{\sf D}({{\sf Bimod }(A-B)})^{op} \times {\sf D}^{+}({\sf Bimod }(C-B)) \rightarrow {\sf D}({\sf Bimod}(C-A)).
$$
Furthermore, when $N^{\bullet} \in {\sf D}^{+}({\sf Bimod }(C-B))$ is a complex of $B$-injectives, then
$$
\operatorname{R}\usHomb^{\bullet}(M^{\bullet},N^{\bullet}) \cong \usHomb^{\bullet}(M^{\bullet},N^{\bullet})
$$
for all $M^{\bullet}$ in ${\sf D}({{\sf Bimod }(A-B)})^{op}$.

\item{}
The functor $\usHomb^{\bullet}(-,-)$ also has a derived functor
$$
\operatorname{R}\usHomb^{\bullet}(-,-):{\sf D}^{-}({{\sf Bimod }(A-B)})^{op} \times {\sf D}({\sf Bimod }(C-B)) \rightarrow {\sf D}({\sf Bimod}(C-A)).
$$
When $M^{\bullet} \in {\sf D}^{-}({{\sf Bimod }(A-B)})^{op}$ is a complex of $B$-projectives then
$$
\operatorname{R}\usHomb^{\bullet}(M^{\bullet},N^{\bullet}) \cong \usHomb^{\bullet}(M^{\bullet},N^{\bullet})
$$
for any $N^{\bullet} \in {\sf D}({\sf Bimod }(C-B))$.

\item{} The two derived functors coincide on
$$
{\sf D}^{-}({{\sf Bimod }(A-B)})^{op} \times {\sf D}^{+}({\sf Bimod }(C-B)).
$$
\end{enumerate}

\end{proposition}

\begin{proof}
In light of Lemma \ref{lemma.lrinjective} and Lemma \ref{lemma.hartshorne}, we may follow the proof of \cite[Theorem 2.2]{amnon}.
\end{proof}

\begin{lemma}
The functor
$$
- \intotimesb -:{\sf Bimod }(A-B) \times {\sf Bimod }(B-C) \rightarrow {\sf Bimod } (A-C)
$$
descends to a bi-$\delta$-functor
$$
\operatorname{Tot }(- \intotimesb -):{\sf K}({\sf Bimod }(A-B)) \times {\sf K}({\sf Bimod }(B-C)) \rightarrow {\sf K}({\sf Bimod } (A-C)).
$$
\end{lemma}

\begin{prop} \label{prop.tensorexist}
The functor
$$
- \intotimesb -:{\sf Bimod }(A-B) \times {\sf Bimod }(B-C) \rightarrow {\sf Bimod }(A-C)
$$
has a derived functor
$$
-\underline{\otimes}_{B}^{\operatorname{L}}-:{\sf D}^{-}({\sf Bimod }(A-B)) \times {\sf D}^{-}({\sf Bimod }(B-C)) \rightarrow {\sf D}^{-}({\sf Bimod } (A-C)).
$$
If either $M^{\bullet}$ is a complex of $A-B$-bimodules which are $B$-projectives or $N^{\bullet}$ is a complex of $B-C$-bimodules which are $B^{op}$-projectives, then
$$
M^{\bullet} \underline{\otimes}_{B}^{\operatorname{L}}N^{\bullet} \cong \operatorname{Tot }(M^{\bullet} \intotimesb N^{\bullet}).
$$
\end{prop}

\begin{proof}
The proof is the same as the proof of \cite[Theorem 2.5]{amnon}.
\end{proof}

\begin{lemma} \label{lemma.extend}
Suppose $B$ is a right Ext-finite, connected $\mathbb{Z}$-algebra.  Then, the torsion functor extends to a left-exact functor $\tau:{\sf Bimod }(A-B) \rightarrow {\sf Bimod }(A-B)$.
\end{lemma}

\begin{proof}
Let $M$ be an $A-B$-bimodule.  We define
$$
\tau(M):= \tau(\bigoplus_{i}e_{i}M)
$$
which, by Lemma \ref{lemma.dlims}, is isomorphic to $\bigoplus_{i}\tau(e_{i}M)$.  If $m \in \tau(e_{i}M)_{j}$ and $a \in A_{lj}$ then $am \in \tau(e_{l}M)$.  The result follows from this.
\end{proof}

By \cite[Theorem 10.5.6]{weibel}, $\tau$ has a right-derived functor
$$
\operatorname{R}\tau: {\sf D}^{+}({\sf Bimod }(A-B)) \rightarrow {\sf D}({\sf Bimod }(A-B)).
$$

\begin{lemma} \label{lemma.tech}
Let $A$ be a right Ext-finite, connected $\mathbb{Z}$-algebra.  If $K^{\bullet}$ is a complex whose terms are free right $A$-modules and $N^{\bullet}$ is a complex of $A-A$-bimodules, then there is a canonical isomorphism
$$
\operatorname{Tot}(K^{\bullet} \intotimes \tau(N^{\bullet})) \cong \tau(\operatorname{Tot}(K^{\bullet} \intotimes N^{\bullet})).
$$
\end{lemma}

\begin{proof}
By Lemma \ref{lemma.dlims}, $\tau$ commutes with direct sums, and the result follows.
\end{proof}

The following is a $\mathbb{Z}$-algebra version of $\mathbb{Z}$-graded local duality \cite[Theorem 5.1(2)]{dualizing}.
\begin{theorem} \label{thm.ld1}
Let $A$ be a right Ext-finite, connected $\mathbb{Z}$-algebra.  Let $M$ be an object in ${\sf Bimod }(K-A)$.  Then
$$
D \operatorname{R}\tau(M) \cong \operatorname{R}\usHom(M, D \operatorname{R} \tau(A))
$$
in ${\sf D}({\sf Bimod }(A-K))$.
\end{theorem}

\begin{proof}
By Proposition \ref{prop.equiv}(1), there exists an injective resolution of $A$ as an $A-A$-bimodule, which we denote by $E^{\bullet}$.  By Lemma \ref{lemma.bproj},  there exists a projective resolution $K^{\bullet}$ of $M$ in ${\sf Bimod }(K-A)$ whose terms are direct sums of modules of form $Ke_i\otimes _ke_jA$.  Since each $Ke_i\otimes _ke_jA$ is a free right $A$-module, we may view $K^{\bullet}$ as a graded right $A$-module resolution of $M=\oplus _ie_iM$ whose terms are free.  Thus, by Lemma \ref{lemma.tech}, $\mbox{Tot}(K^{\bullet} \intotimes E^{\bullet})$ is a $\tau$-acyclic complex.  Furthermore, this complex is quasi-isomorphic to $M \intotimes A \cong M$.

By Lemma \ref{lemma.tech}, we have
\begin{eqnarray*}
M \intotimes^{\operatorname{L}} \operatorname{R}\tau(A) & \cong & \mbox{Tot}(K^{\bullet} \intotimes \tau(E^{\bullet})) \\
& \cong & \tau(\mbox{Tot}(K^{\bullet} \intotimes E^{\bullet})) \\
& \cong & \operatorname{R}\tau(M).
\end{eqnarray*}

Thus,
\begin{eqnarray*}
\operatorname{R}\usHom(M, D\operatorname{R} \tau(A)) & \cong & \usHom^{\bullet}(K^{\bullet},D \operatorname{R}\tau(A)) \\
& \cong & D(\operatorname{Tot}(K^{\bullet} \intotimes \operatorname{R}\tau(A)) \\
& \cong & D(M \intotimes^{\operatorname{L}} \operatorname{R}\tau(A)) \\
& \cong & D\operatorname{R}\tau(M),
\end{eqnarray*}
where the second map is from Corollary \ref{cor.D}.
\end{proof}



\section{AS-regularity and ASF-regularity} \label{section.regularity}
In this section, we assume $A$ is a connected $\mathbb{Z}$-algebra.

\begin{definition}
$A$ is {\it AS-regular of dimension $d$ and of Gorenstein parameter $l$} if
\begin{enumerate}
\item{} $\operatorname{pd} e_{i}A_{0} = d$ for all $i \in \mathbb{Z}$, and

\item{} $\sExt^{q}(e_{i}A_{0}, e_{j}A) \cong \begin{cases} A_{ii}^{*} \mbox{ if $q=d$ and $j=i+l$} \\ 0 \mbox{ otherwise.} \end{cases}$
\end{enumerate}
\end{definition}
As pointed out in \cite[Section 4.1]{quadrics}, $A$ is a $\mathbb{Z}$-graded AS-regular algebra of dimension $d$ and of Gorenstein parameter $l$ if and only if its associated $\mathbb{Z}$-algebra $\bigoplus_{i,j \in \mathbb{Z}}A_{j-i}$ is AS-regular of dimension $d$ and of Gorenstein parameter $l$.  Moreover, there exist AS-regular $\mathbb{Z}$-algebras of dimension $3$ and Gorenstein paramter $4$ which do not arise in this manner \cite[Section 5.6]{quadrics}.  In \cite{abstractpn}, we will give a general method for constructing AS-regular $\mathbb{Z}$-algebras.

Since $A$ is positively graded, $A^{op}$ is no longer positively graded, but, although $A^{op}$ is not isomorphic to $\widetilde{A^{op}}$, $A^{op}$ and $\widetilde{A^{op}}$ have equivalent graded left and right module categories by Proposition \ref{prop.equiv}.  This motivates the following definition.

\begin{definition}
We say the negatively graded algebra $A^{op}$ is {\it AS-regular of dimension $d$ and of Groenstein parameter $l$} if $\widetilde{A^{op}}$ is AS-regular of dimension $d$ and of Gorenstein parameter $l$.
\end{definition}

\begin{prop} \label{prop.someext}
Suppose $A$ is an Ext-finite, AS-regular algebra of dimension $d$ and of Gorenstein parameter $l$.  If $n \geq 1$ and $N$ is in ${\sf Gr }A$, then $\usExt^{i}(A/A_{\geq n},N)=0$ for $i>d$ and $\usExt^{i}(A/A_{\geq n},e_{m}A)=0$ for $i \neq d$.
\end{prop}

\begin{proof}
We prove the result by induction on $n$.  When $n=1$, the first result follows from the fact that the $j$th component of $\usExt^{i}(A/A_{\geq 1},N)$ may be computed from a length $d$ projective resolution of the $j$th component of $A/A_{\geq 1}$, while the second result follows directly from regularity.

For the general case, we note that the exact sequence
$$
0 \rightarrow A_{\geq n}/A_{\geq n+1} \rightarrow A/A_{\geq n+1} \rightarrow A/A_{\geq n} \rightarrow 0
$$
in ${\sf Bimod }(A-A)$ induces a long exact sequence, of which
$$
\usExt^{i}(A/A_{\geq n}, N) \rightarrow \usExt^{i}(A/A_{\geq n+1},N) \rightarrow \usExt^{i}(A_{\geq n}/A_{\geq n+1},N)
$$
is a part.  If $i>d$ the left term is zero by induction while the right term is zero by Corollary \ref{cor.extstar} (which applies by Remark \ref{remark.ext}) and regularity.  Therefore, the center is zero in this case.  If $i<d$ and $N=e_{m}A$, the same reasoning ensures that the center is zero.
\end{proof}

The proof of the next result, Corollary \ref{cor.gor0}, generalizes that of \cite[Corollary 7.3]{abstract}, which was used to prove Serre vanishing for noncommutative projective lines.  In this paper, we use Corollary \ref{cor.gor0} to prove Corollary \ref{cor.af}, which, in turn, is utilized in the proof of the main result in \cite{abstractpn}.  In the proof of Corollary \ref{cor.gor0}, we will use the following terminology.  Suppose $\lambda, \rho \in \mathbb{Z}$ with $\lambda \leq \rho$ and let $N \in {\sf Gr }A$.  We write $N \subset [\lambda, \rho]$ if $N_{i}$ nonzero implies that $\lambda \leq i \leq \rho$.  We say $N$ is {\it concentrated in degree $m$} if $N \subset [m,m]$.

\begin{cor} \label{cor.gor0}
Suppose $A$ is an Ext-finite AS-regular algebra of dimension $d$ and Gorenstein parameter $l$, and let $N$ be an object of ${\sf Gr }A$.
\begin{enumerate}
\item{}
For $i>d$,
$$
\mbox{R}^{i}\tau (N)=0.
$$
\item{} For $i \neq d$ and all $j$,
$$
\mbox{R}^{i}\tau (e_{j}A)=0.
$$
\item{} For all $i,j$,
$$
(\mbox{R}^{d}\tau (e_{j}A))_{j-l-i} \cong A^{*}_{j-l-i, j-l}
$$
as right $A_{j-l-i, j-l-i}$-modules.

\end{enumerate}
\end{cor}

\begin{proof}
By Lemma \ref{lemma.torsfunct} and the fact that ${\sf Gr }A$ has exact direct limits, the first two results follow directly from Proposition \ref{prop.someext}.

To prove (3), we prove three preliminary claims.  We first claim that
$$
\usExt^{d}(A/A_{\geq n+1}, e_{j}A) \subset [j-l-n,j-l].
$$
To this end, we note that, by Proposition \ref{prop.someext}, the sequence
$$
0 \rightarrow A_{\geq n}/A_{\geq n+1} \rightarrow A/A_{\geq n+1} \rightarrow A/ A_{\geq n} \rightarrow 0
$$
induces an exact sequence
\begin{equation} \label{eqn.ext2}
0 \rightarrow \usExt^{d}(A/ A_{\geq n}, e_{j}A) \rightarrow \usExt^{d}(A/A_{\geq n+1}, e_{j}A) \rightarrow \usExt^{d}(A_{\geq n}/A_{\geq n+1}, e_{j}A) \rightarrow 0.
\end{equation}
If $n=0$, then the claim follows from AS-regularity.  By Corollary \ref{cor.extstar} and AS-regularity, the third term in (\ref{eqn.ext2}) is concentrated in degree $j-l-n$, so the general case follows from the induction hypotheses and (\ref{eqn.ext2}).

By (\ref{eqn.ext2}), there is an injection
$$
\usExt^{d}(A/ A_{\geq i+1}, e_{j}A) \rightarrow \dlim \usExt^{d}(A/ A_{\geq n}, e_{j}A).
$$

We next claim that this injection is surjective in degree $j-l-i$.  To prove this claim, we note that if $n<i$, $\usExt^{d}(A/ A_{\geq n+1}, e_{j}A)_{j-l-i}=0$ by the first claim.  On the other hand, if $n \geq i$, the canonical map
\begin{equation} \label{eqn.newst}
\usExt^{d}(A/ A_{\geq i+1}, e_{j}A)_{j-l-i} \rightarrow \usExt^{d}(A/ A_{\geq n+1}, e_{j}A)_{j-l-i}
\end{equation}
is an isomorphism, as follows.  As in the proof of the first claim, the third term in (\ref{eqn.ext2}) is concentrated in degree $j-l-n$, so that
\begin{equation} \label{eqn.newstprime}
\usExt^{d}(A/ A_{\geq n}, e_{j}A)_{\geq j-l-n+1} \rightarrow \usExt^{d}(A/ A_{\geq n+1}, e_{j}A)_{\geq j-l-n+1}
\end{equation}
is an isomorphism.  Therefore, the maps
\begin{equation} \label{eqn.newstnew}
\usExt^{d}(A/ A_{\geq i+1}, e_{j}A)_{j-l-i} \rightarrow \usExt^{d}(A/ A_{\geq i+2}, e_{j}A)_{j-l-i} \rightarrow
\end{equation}
$$
\cdots \rightarrow \usExt^{d}(A/ A_{\geq n+1}, e_{j}A)_{j-l-i}
$$
are isomorphisms for every $n \geq i$ and the second claim follows.

We finally claim that $\usExt^{d}(A/A_{\geq n+1}, e_{j}A)_{j-l-n} \cong A_{j-l-n, j-l}^{*}$.  To prove this, we note that when $n=0$, the claim follows from AS-regularity of $A$.  For $n>0$,
\begin{eqnarray*}
\usExt^{d}(A/A_{\geq n+1}, e_{j}A)_{j-l-n} & \cong & \usExt^{d}(A_{\geq n}/A_{\geq n+1}, e_{j}A)_{j-l-n} \\
& \cong & A_{j-l-n, j-l}^{*}
\end{eqnarray*}
where the first isomorphism follows from the first claim and (\ref{eqn.ext2}), while the second isomorphism follows from Corollary \ref{cor.extstar} and AS-regularity.

Next, we prove (3).  We have
\begin{eqnarray*}
(\mbox{R}^{d}\tau (e_{j}A))_{j-l-i} & \cong & \dlim \usExt^{d}(A/ A_{\geq n}, e_{j}A)_{j-l-i} \\
& \cong & \usExt^{d}(A/ A_{\geq i+1}, e_{j}A)_{j-l-i} \\
& \cong & A^{*}_{j-l-i, j-l},
\end{eqnarray*}
where the first isomorphism is from Lemma \ref{lemma.torsfunct}, the second isomorphism follows from the second claim, and the third isomorphism follows from the third claim.
\end{proof}

\begin{definition}
A connected $\mathbb{Z}$-algebra $A$ is called {\it ASF-regular of dimension $d$ and of Gorenstein parameter $l$} if for all $j \in \mathbb{Z}$,
\begin{enumerate}
\item{} $\mbox{sup}\{\operatorname{pd}e_{j}A_{0}| j \in \mathbb{Z}\} = d< \infty$, and

\item{} $\operatorname{R}^{q}\tau(e_{j}A) = \begin{cases} D(Ae_{j-l}) \mbox{ if $q=d$} \\ 0 \mbox{ otherwise} \end{cases}$
as graded right $A$-modules.
\end{enumerate}
\end{definition}
As above, we say that $A^{op}$ is ASF-regular of dimension $d$ and of Gorenstein parameter $l$ if $\widetilde{A^{op}}$ is.

\begin{corollary} \label{cor.af}
Let $A$ be an Ext-finite AS-regular algebra of dimension $d$ and Gorenstein parameter $l$ such that $A_{ii}=k$ for all $i$.  Then, for every $j \in \mathbb{Z}$, there is an isomorphism of graded right $A$-modules
$$
\operatorname{R}^{q}\tau(e_{j}A) \cong \begin{cases} D(Ae_{j-l}) \mbox{ if $q=d$} \\ 0 \mbox{ otherwise.} \end{cases}
$$
\end{corollary}

\begin{proof}
In light of Lemma \ref{lemma.usual}, the isomorphism is, componentwise, just that in Corollary \ref{cor.gor0}.  It thus remains to check that the isomorphism is compatible with right multiplication.  To this end, we let $i<m$ be nonnegative integers and we let $a \in A_{j-l-m, j-l-i}$.  We must show the diagram
$$
\begin{CD}
\dlim \usExt^{d}(A/ A_{\geq n}, e_{j}A)_{j-l-i} & \longrightarrow & A^{*}_{j-l-i, j-l} \\
@AAA @AAA \\
\dlim \usExt^{d}(A/ A_{\geq n}, e_{j}A)_{j-l-m} & \longrightarrow & A^{*}_{j-l-m, j-l}
\end{CD}
$$
with verticals induced by left multiplication by $a$ and horizontals from Corollary \ref{cor.gor0}, commutes.  From the second claim in the proof of Corollary \ref{cor.gor0}, it thus suffices to show
$$
\begin{CD}
\usExt^{d}(A/ A_{\geq i+1}, e_{j}A)_{j-l-i} & \longrightarrow & A^{*}_{j-l-i, j-l} \\
@AAA @AAA \\
\usExt^{d}(A/ A_{\geq m+1}, e_{j}A)_{j-l-m} & \longrightarrow & A^{*}_{j-l-m, j-l}
\end{CD}
$$
with maps as in the first diagram, commutes.  By the proof of the third claim in the proof of Corollary \ref{cor.gor0}, it suffices to show the induced diagram
$$
\begin{CD}
\usExt^{d}(A_{\geq i}/ A_{\geq i+1}, e_{j}A)_{j-l-i} & \longrightarrow & A^{*}_{j-l-i, j-l} \\
@AAA @AAA \\
\usExt^{d}(A_{\geq m}/ A_{\geq m+1}, e_{j}A)_{j-l-m} & \longrightarrow & A^{*}_{j-l-m, j-l}
\end{CD}
$$
commutes.  Using the proof of Corollary \ref{cor.extstar}, it suffices to show that the diagram
$$
\begin{CD}
\sExt^{d}(A_{j-l-i,j-l}\otimes e_{j-l}A_{0}, e_{j}A) & \longrightarrow & A^{*}_{j-l-i, j-l} \\
@AAA @AAA \\
\sExt^{d}(A_{j-l-m,j-l} \otimes e_{j-l}A_{0}, e_{j}A) & \longrightarrow & A^{*}_{j-l-m, j-l}
\end{CD}
$$
with verticals induced by multiplication by $a$ and with horizontal maps from Lemma \ref{lemma.extcommute}, commutes. Finally, this commutes by Corollary \ref{cor.extstar} and the universal property of universal $\delta$-functors since, by the remark preceding Lemma \ref{lemma.extcommute}, the functors
$$
\usExt^{d}(A_{j-l-m,j-l} \otimes e_{j-l}A_{0}, -),
$$
$$
\usExt^{d}(A_{j-l-i,j-l} \otimes e_{j-l}A_{0}, -),
$$
$$
\usExt^{d}(e_{j-l}A_{0}, -)\otimes A_{j-l-m,j-l}^{*},
$$
$$
\usExt^{d}(e_{j-l}A_{0}, -)\otimes A_{j-l-i,j-l}^{*}
$$
are components of universal $\delta$-functors.
\end{proof}

\begin{cor} \label{cor.asfas}
If $A$ is Ext-finite, $A_{ii}=k$ for all $i$, and $A$ is AS-regular of dimension $d$ and of Gorenstein parameter $l$, then $A$ is ASF-regular of dimension $d$ and of Gorenstein parameter $l$.
\end{cor}

\begin{proof}
The first condition for ASF-regularity is immediate while the second condition for ASF-regularity holds by Corollary \ref{cor.af}.
\end{proof}
We do not know of an example of an ASF-regular $\mathbb{Z}$-algebra which is not AS-regular.  In fact, we will show that these concepts are equivalent assuming the existence of a balanced dualizing complex (Theorem \ref{theorem.main2}).

For the remainder of this section, $\tilde{\tau}:{\sf Gr }\widetilde{A^{op}} \rightarrow {\sf Gr }\widetilde{A^{op}}$ denotes the torsion functor.
The following theorem is our version of \cite[Theorem 3.3]{jorgensen} (see Remark \ref{remark.jorg}).

\begin{theorem} \label{theorem.main}
Suppose $A$ is a connected, coherent $\mathbb{Z}$-algebra.  Suppose, further, that both $A$ and $A^{op}$ are ASF-regular of dimension $d$ and of Gorenstein parameter $l$.  Then
$$
\operatorname{R}\usHom(-,D\operatorname{R}\tau(A)):{\sf D}^{b}({\sf coh }A) \leftrightarrow {\sf D}^{b}({\sf coh }\widetilde{A^{op}}):\operatorname{R}\usHomao(-,D\operatorname{R}\tilde{\tau}(\widetilde{A^{op}}))
$$
is a duality.
\end{theorem}

\begin{proof}
We first note that since $A$ is connected and coherent, $A$ is right and left Ext-finite by Lemma \ref{lemma.res}.  By Theorem \ref{thm.ld1}, $\operatorname{R}\usHomao(D\operatorname{R}\tau(e_{i}A),D\operatorname{R}\tilde{\tau}(\widetilde{A^{op}}))$ is isomorphic to
$$
\operatorname{R}\usHomt(\operatorname{R}\usHom(e_{i}A, D \operatorname{R}\tau(A)),D\operatorname{R} \tilde{\tau}(\widetilde{A^{op}})).
$$
On the other hand, by Remark \ref{remark.ext}, ASF-regularity, and Theorem \ref{thm.ld1},
\begin{eqnarray*}
\operatorname{R}\usHomt(D\operatorname{R}\tau(e_{i}A),D\operatorname{R}\tilde{\tau}(\widetilde{A^{op}}))
& \cong & \operatorname{R}\usHomt(Ae_{i-l}[d],D\operatorname{R}\tilde{\tau}(\widetilde{A^{op}})) \\
& \cong & \operatorname{R}\usHomt(\tilde{e}_{-i+l}\widetilde{A^{op}}[d],D\operatorname{R}\tilde{\tau}(\widetilde{A^{op}})) \\
& \cong & \operatorname{R}\usHomt(\tilde{e}_{-i+l}\widetilde{A^{op}},D\operatorname{R}\tilde{\tau}(\widetilde{A^{op}}))[-d] \\
& \cong & D \operatorname{R}\tilde{\tau}(\tilde{e}_{-i+l}\widetilde{A^{op}})[-d] \\
& \cong & \widetilde{A^{op}}\tilde{e}_{-i} \\
& \cong & e_{i}A
\end{eqnarray*}
in ${\sf D}^{b}({\sf coh }A)$ for every $i \in \mathbb{Z}$.  Since every object in ${\sf D}^{b}({\sf coh }A)$ has a finite length, finitely generated free resolution by Lemma \ref{lemma.res}, Proposition \ref{prop.gldim}, and Corollary \ref{cor.min}, the result follows.
\end{proof}

\begin{remark} \label{remark.jorg}
By Corollary \ref{cor.asfas}, the conclusion of Theorem \ref{theorem.main} holds if $A_{ii}=k$ for all $i$ and the ASF-regularity hypothesis is replaced by AS-regularity.
\end{remark}

\begin{theorem} \label{theorem.main2}
Suppose $A$ is an Ext-finite connected $\mathbb{Z}$-algebra such that $A_{ii}=k$ for all $i$, and such that $\operatorname{R}\tau(A) \cong \operatorname{R}\tilde{\tau}(\widetilde{A^{op}})$ in ${\sf D}({\sf Bimod }(A-A))$.  Then the following are equivalent:
\begin{enumerate}
\item{} $A$ is AS-regular of dimension $d$ and of Gorenstein parameter $l$.

\item{} $A^{op}$ is AS-regular of dimension $d$ and of Gorenstein parameter $l$.

\item{} $A$ is ASF-regular of dimension $d$ and of Gorenstein parameter $l$.

\item{} $A^{op}$ is ASF-regular of dimension $d$ and of Gorenstein parameter $l$.
\end{enumerate}
\end{theorem}

\begin{proof}
The fact that (1) implies (3) (and (2) implies (4)) is Corollary \ref{cor.asfas}.  We now prove (3) implies (2).  We note that since $A$ is right Ext-finite, the proof of Lemma \ref{lemma.extend} implies that if $M$ is an $A-A$-bimodule, then $e_{i}\operatorname{R}\tau(M) \cong \operatorname{R}\tau(e_{i}M)$ for all $i \in \mathbb{Z}$.  Therefore, by ASF-regularity,
\begin{eqnarray*}
(D\operatorname{R}\tilde{\tau}(\widetilde{A^{op}}))e_{j} & \cong & \tilde{e}_{-j}(D\operatorname{R}\tilde{\tau}(\widetilde{A^{op}})) \\
& \cong & D(e_{j}\operatorname{R} \tilde{\tau}(\widetilde{A^{op}})) \\
& \cong & D(e_{j}\operatorname{R}\tau(A)) \\
& \cong & D(\operatorname{R}\tau(e_{j}A)) \\
& \cong & Ae_{j-l}[d]
\end{eqnarray*}
in ${\sf D}(A-{\sf Gr })$, where we have implicitly used Proposition \ref{prop.equiv} in the first two isomorphisms.  This implies that $\tilde{e}_{j}(D\operatorname{R}\tilde{\tau}(\widetilde{A^{op}})) \cong \tilde{e}_{j+l}\widetilde{A^{op}}[d]$ in ${\sf D}({\sf Gr }\widetilde{A^{op}})$.


Since $A$ is left Ext-finite, $\widetilde{A^{op}}$ is right Ext-finite by Lemma \ref{lemma.extagain}, so that Theorem \ref{thm.ld1} implies (in the third isomorphism below) that
\begin{eqnarray*}
\operatorname{R}\sHomt(\tilde{e}_{i}(\widetilde{A^{op}})_{0},\tilde{e}_{j}\widetilde{A^{op}}) & \cong & \operatorname{R}\sHomt(\tilde{e}_{i}(\widetilde{A^{op}})_{0},\tilde{e}_{j-l}(D\operatorname{R}\tilde{\tau}(\widetilde{A^{op}}))[-d]) \\
& \cong & \tilde{e}_{j-l}\operatorname{R}\usHomt(\tilde{e}_{i}(\widetilde{A^{op}})_{0}, (D\operatorname{R}\tilde{\tau}(\widetilde{A^{op}})))[-d] \\
& \cong & \tilde{e}_{j-l}D \operatorname{R}\tilde{\tau}(\tilde{e}_{i}(\widetilde{A^{op}})_{0})[-d] \\
& \cong & \tilde{e}_{j-l}(D \tilde{e}_{i}(\widetilde{A^{op}})_{0})[-d] \\
& \cong & \begin{cases} k[-d] & \mbox{ if $j=i+l$} \\ 0 & \mbox{ otherwise.}\end{cases}
\end{eqnarray*}
By Corollary \ref{cor.sbalance}, $\operatorname{pd }e_{i}(\widetilde{A^{op}})_{0} \leq d$, while by the calculation above, $\operatorname{pd }e_{i}(\widetilde{A^{op}})_{0} \geq d$.  This establishes that (3) implies (2), as well as the fact that (4) implies (1).
\end{proof}




\end{document}